\documentclass[preprint]{article}
\usepackage{lineno}
\usepackage{tikz}
\usetikzlibrary{arrows}
\usepackage{graphicx}
\usepackage{color}
\usepackage[colorlinks]{hyperref}
\usepackage{euscript}
\usepackage{mathdots}
\usepackage{yhmath}
\usepackage{cancel}
\usepackage{siunitx}
\usepackage{array}
\usepackage{multirow}
\usepackage{gensymb}
\usepackage{tabularx}
\usepackage{booktabs}
\usepackage{acronym}
\usepackage{fixltx2e}
\usepackage{enumitem}
\usepackage{upgreek}
\usepackage[all]{xy}
\usepackage{mathtools}
\usepackage[normalem]{ulem}
\usepackage{amsmath,amsthm,amsfonts,amssymb,amscd,euscript}

\newcommand{\Z}{\mathbb Z}
\newcommand{\N}{\mathbb N}

\newcommand{\F}{{\bf{F}}}
\newcommand{\ifs}{$\mathfrak{I}=(X,\,\mathcal{F},\,\Sigma)\;$}

\newtheorem{thm}{Theorem}[section]
\newtheorem{cor}[thm]{Corollary}
\newtheorem{lem}[thm]{Lemma}

\newtheorem{prop}[thm]{Proposition}
\newtheorem{example}[thm]{Example}
\newtheorem{defn}[thm]{Definition}
\theoremstyle{remark}
\newtheorem{rem}[thm]{Remark}

\modulolinenumbers[5]

\begin{document}

	\title{Shift limits of a non-autonomous system }
	\author{Dawoud Ahmadi Dastjerdi\footnote{dahmadi1387@gmail.com (ahmadi@guilan.ac.ir)},\and Mahdi Aghaee \footnote{mahdi.aghaei66@gmail.com}
	}
	\maketitle

\begin{abstract}
Let $t=t_1t_2\cdots$ be an element of the full shift with shift map $\tau$  on a finite set of characters $\mathcal{A}$ and
let $ \Sigma=\text{ closure} \{\tau^i(t):\;i\in\N\cup\{0\}\}$.	
Let $f_t=f_{t_1,\,\infty}=\cdots\circ f_{t_2}\circ f_{t_1} $ be a non-autonomous system over a compact metric space $ X $ where $t_i\in \mathcal A $. 
The set $\F_t^+=\{f_{\tau^i(t)}:\; i\in\N\}$ is called the shifted family of $f_t$.
If $t$ is a transitive point of  the full shift on $\mathcal A$, then by introducing a natural topology, $\overline{\F_t^+}$ is a classical IFS; otherwise, $\overline{\F_t^+}=\{f_\sigma=f_{\sigma_1,\,\infty}:\; \sigma\in\Sigma\}$ is a generalized IFS.
	We will show that if $ f_t$ has some various shadowing and specification properties,  then this is true for  $f_{\sigma}\in\overline{\F^+_t}$; however, this claim is not true for other properties such as transitivity, mixing and exactness. 
	Also, if $ \Sigma $ is sofic and 
$x\in X$ is periodic point for some $f_\sigma\in\overline{\F^+_t}$, then  there is a periodic $\sigma'\in\Sigma$ such that $x$ is periodic for   $f_{\sigma'}\in\overline{\F^+_t}$.
\end{abstract}



\section{Introduction}
Assume that the functions defining a  non-autonomous dynamical system are coming from $\mathcal{F}=\{f_0,\ldots, f_{k-1}\}$. Then we write such a system as $f_ {t_1,\,\infty}=\cdots f_{t_2}\circ f_{t_1}$  where $t_i\in\mathcal{A}=\{0,\,1,\ldots,k-1\}$ (our priority is to have $\mathcal{A}$ finite, but if our results is valid for infinite case, we do mention explicitly).
 On the other hand, $t=t_1t_2\cdots\in\mathcal{A}^\N$ can be considered as an element of the subshift $\Sigma=\overline{\mathcal{O}^+(t)}=$closure$\{\tau^i(t):\,i\in\N_0=\N\cup\{0\}\}$ where $\tau$ is the shift map and $\mathfrak{I}=(X,\,\mathcal{F},\,\Sigma)$ is a \emph{general iterated function system} (IFS).
  That means the combinations of maps are not freely as in the \emph{classical IFS} and are done according to restrictions imposed by $\Sigma$.
   By this settings a classical IFS is $\mathfrak{I}=(X,\,\mathcal{F},\,\Sigma_{|\mathcal{F}|})$ where $\Sigma_{|\mathcal{F}|}$ is the full shift on $|\mathcal{F}|$ symbols. 
   Here for $\sigma=\sigma_1\sigma_2\cdots\in\Sigma$, we call the non-autonomous dynamical system $f_{\sigma_1,\,\infty}$ a \emph{shift limit} of $f_ {t_1,\,\infty}$ and we are interested to see what dynamical properties of $f_ {t_1,\,\infty}$ are shared with $f_ {\sigma_1,\,\infty}$. 
Some properties that we have considered  are specification, shadowing, chaos and in particular different periodic behavior  appearing in non-autonomous dynamical systems. 
Our approach is to consider a typical transitive $t$ in a subshift $\Sigma=\overline{\mathcal{O}^+(t)}$ and to see how the aforementioned  dynamical properties of 
$f_{t_1,\,\infty}$ is carried to $f_{\sigma_1,\,\infty}$ for $\sigma=\sigma_1\sigma_2\cdots\in\Sigma$. 
If that does so, then we say that specific property \emph{along} $t$ is preserved along $\sigma$ as well.

Let us recall that many aspects of both dynamical systems, that is non-autonomous dynamical systems and IFS's have been worked out in other investigations. 
Salman and Das \cite{salman2020specification} discussed specification property in non-autonomous discrete systems and Vasisht and Das \cite{vasisht2021specification} studied specification and shadowing properties
for non-autonomous Systems. Pravec \cite{pravec2019remarks} considered some different periodic points for
non-autonomous dynamical systems.
Specification and shadowing for classical iterated function systems have been studied in \cite{rodrigues2016specification,bahabadi2015shadowing}. The general IFS was considered by the authors \cite{dastjerdi2022iterated} and so fairly new subject.
There, a good deal of attention has been paid to transitivity. For transitivity of non-autonomous systems, 
one can consult \cite{vasisht2020furstenberg, salman2020multi, sanchez2017chaos}.
It is worth to mention that we learned a lot from 
Kolyada and Snoha \cite{kolyada1996topological}. However, their main interest is on entropy that we have not considered here.

\vspace{.5cm}
	\noindent{\bf The present paper is organized as follows.} After introducing some required basic notions in Preliminaries,
 in Section \ref{Sp and shadowing}, we show that specification and shadowing along a transitive orbit of a subshift,
 acting as our starting non-autonomous system, will be preserved along all other shift limits.
In Section \ref{sec Periodicity in IFS }, after recalling different types of periodicity in a non-autonomous system, we see how these different periodic behaviors are projected to shift limits.

\section{Preliminaries}\label{sec preliminaries}
Throughout the paper, $ (X,\,d) $ will be a compact metric space and $\mathcal{A}$ is a set of finite characters . 
Also $\N$ will be the set of natural numbers and  ${\bf N}$ the set of all non-autonomous systems (NDS) on $X$. We are dealing with those NDS's whose functions are in $\{f_j:X\rightarrow X:\; f_j$ continuous $j\in\mathcal{A}\}$.
So for $t=t_1t_2\cdots$ where $t_i\in \mathcal{A}$, denote by 

\begin{equation}\label{eq NDS}
	f_{t_1,\,\infty}=\cdots\circ f_{t_2}\circ f_{t_1}
\end{equation}
 the NDS along  $t$  and let $ f_{t_i,\,\infty}=\cdots\circ f_{t_{i+1}}\circ f_{t_i} $ be the $i$\,th shifted 
NDS of $f_{t_1,\,\infty}$. Set
\begin{equation}\label{eq F}
	\F^+_t=\F^+_t(f_{t_1,\,\infty}) =\{f_{t_i,\,\infty}:\; i\in\N\}\subseteq {\bf N}
\end{equation}
to be the \emph{shifted family} of the non-autonomous system $f_{t_1,\infty}$. Let $D(f,\,g)=\sup_{x\in X}d(f(x),\,g(x))$ be
the metric on $C(X)$, the set of continuous  self maps on $X$.
 Then 
  $$\rho(f_{t_1,\,\infty},\,f_{t'_1,\,\infty})=
 	\sum_{n\in\N}
 	\frac{D(f_{t_1\cdots t_n},\,f_{t'_1\cdots t'_n})}{2^{n-1}}  $$ 
is a metric on ${\bf N}$ where 
$ f_{t_1\cdots t_n}=f_{t_n}\circ \cdots\circ f_{t_1}$. An NDS is injective or surjective if each function defining the NDS are injective or surjective respectively.

	Now let $(\Sigma,\,\tau)$ be the one sided subshift on $\mathcal{A}$ whose shift map is $\tau$ ( See \S \ref{sec symbolic} for a brief review of symbolic dynamics) and equip $(\Sigma\times{\bf N})$ with the product topology.
	We are interested in the dynamics of the space of skew product ${\bf S}\subset (\Sigma\times{\bf N})$ 
	\begin{equation}\label{eq bf S}
		{\bf F}:{\bf S}\to{\bf S}
	\end{equation}
	where ${\bf F}$ is defined as 
	${\bf F}(\sigma,\,f_{\sigma_1,\,\infty})=
	(\tau(\sigma),\,f_{\tau(\sigma)_1,\,\infty})$. Note that by the above settings, ${\bf S}$ is compact and ${\bf F}$ is continuous.
	When $t\in \mathcal{A}^\N$, then 
	$\Sigma=\overline{\mathcal{O}^+(t)}$ is a  subshift obtained by the orbit of $t$ and ${\bf F}^+_t=\text{pr}_2(\mathcal{O}^+(t,\,f_{t_1,\,\infty}))$. This means ${\bf F}^+_t$ is the projection on the second component of the orbit of the point $(t,\,f_{t_1,\,\infty})$ under ${\bf F}$.
	The following is immediate.

\begin{prop}\label{prop closure}
	$\overline{\F^+_{t}}=\{f_{\sigma_1,\,\infty}:\;\sigma=\sigma_1\sigma_2\cdots\in\overline{\mathcal{O}^+(t)}\}$.	
\end{prop}
The above proposition demonstrates a close relation between an NDS and the theory of iterated function systems on a general subshift which we will give a brief introduction in the next subsection.
\subsection{``Generalized" iterated function systems.}
The classical  \emph{iterated function system} (IFS) consists of all possible combinations of finitely many continuous self maps $ \mathcal{F}=\left\lbrace f_0,\,\ldots,\,f_{k-1} \right\rbrace  $ on  $X$. 
In other words, a classical IFS consists of $X$ and all maps such as 
\begin{equation}\label{eq def of f_u}
	f_u=f_{u_n}\circ\cdots \circ f_{u_1}:X\rightarrow X,\quad u=u_1\cdots u_n\in\mathcal{L}(\Sigma_k)
\end{equation}
where $\Sigma_{k}$ is the full shift on $\mathcal{A}=\{0,\,1,\cdots,k-1\}$ and $\mathcal{L}(\Sigma_k)$ is the language of $\Sigma_k$. (See below for a brief introduction of symbolic dynamics.) If one replaces $\Sigma_k$ with a subshift $\Sigma\subseteq \Sigma_k$ in \eqref{eq def of f_u}, then it will be a ``generalized" iterated function system which we refer to it again as IFS and it is denoted by
\begin{equation}\label{eq IFS}
	\mathfrak{I}=(X,\,  \mathcal{F},\, \Sigma),
\end{equation}
where $\mathcal{F}=\left\lbrace f_0,\,\ldots,\, f_{k-1}\right\rbrace$.

When $\mathfrak{I}$ is the  classical IFS; or equivalently,  when $\Sigma$ is a full shift, then still we use $\eqref{eq IFS}$ but  the subshift will be shown by $\Sigma_{|\mathcal{F}|}$.   The classical IFS is usually denoted by  $(X,\,\mathcal{F})$ by other authors.

The \emph{forward orbit} of a point $x\in X$ in an IFS, is defined as $\mathcal{O}^+(x)=\{f_u(x):\;u\in\Sigma\}$. Also to each $\sigma=\sigma_1\sigma_2\cdots\in\Sigma$, a NDS $f_\sigma=f_{\sigma_1,\,\infty}$ will be associated that  it is called the \emph{IFS along $\sigma$}.
The trajectory (or orbit) of $ x $ under $ f_\sigma$ is given by the sequence 
\begin{equation*}
	\mathcal{O}^+(x)=\{x,\, f_{\sigma_1}(x),\,f_{\sigma_1\sigma_{2}}(x),\,\ldots,\, f_{\sigma_1\cdots\sigma_{n}}(x), \,\ldots \}
\end{equation*}
where $ f_{\sigma_1\cdots \sigma_n}=f_{\sigma_n}\circ \cdots\circ f_{\sigma_1}$, $ n\geq 1 $. 
In general, $ f_{\sigma_m\cdots \sigma_n}=f_{\sigma_n}\circ \cdots\circ f_{\sigma_m}$, $ n\geq m $. 
Notice that when $ \mathcal{F}$ consists of just one function, then $ f_\sigma $ is a ``conventional" discrete dynamical system. For further properties of an IFS see 
\cite{dastjerdi2022iterated}.

\subsection{Symbolic dynamics}\label{sec symbolic}
Since we are dealing with a general subshift in an IFS such as  \eqref{eq IFS}, we require a  brief recall of the symbolic dynamics. A basic source for this theory is \cite{lind2021introduction}; here we just give a brief recall.
 Let $\mathcal{A}=\{0,\,1,\cdots,k-1\}$ be the set of finite characters also called alphabet.
  Let  $\Sigma_{k}= \mathcal{A}^{\Z} $ (resp.  $ \mathcal{A}^{\N}  $) be the collection of all bi-infinite (resp. right-infinite) sequences of
   symbols from $ \mathcal{A} $.
    The map $\tau:\Sigma_k\rightarrow \Sigma_k$ defined by $\tau(\sigma)_i=\sigma_{i+1}$ is called the \emph{shift map} and the pair $ (\Sigma_k,\,\tau) $ is the \emph{full shift} on $k$ symbols.
     Any invariant closed subspace $\Sigma$ of $\Sigma_k$ is called a \emph{subshift}.

 A \emph{word} or a \emph{block} over $ \mathcal{A} $ is a finite sequence of symbols from $ \mathcal{A} $. 
 Denote by $ \mathcal{L}_{n}(\Sigma) $ 
 the set of all \emph{admissible} words of length $n$ and call 
 $ \mathcal{L}(\Sigma): =\bigcup_{n=0}^{\infty}\mathcal{L}_n(\Sigma)$ 
 the \emph{language} of $ \Sigma $. 
 A subshift is characterized by its language or its forbidden words.
 For $ u \in \mathcal{L}_k(\Sigma) $,
 let the cylinder $_\ell[u]_{\ell+k-1}$$= _\ell\hspace{-1mm}[u_\ell\cdots u_{\ell+k-1}]_{\ell+k-1} $ be the set $ \{\sigma=\cdots\sigma_{-1}\sigma_0\sigma_1\cdots \in \Sigma :
 \sigma_{\ell}\cdots\sigma_{\ell+k-1} = u\} $.
 If $\ell=0$, we drop the subscripts and we just write $[u]$.
  A point $ \sigma \in \Sigma $ is \emph{transitive} if every word in $ \Sigma $ appears in $ \sigma $ infinitely many often and a 
  subshift $ \Sigma $ is \emph{irreducible} iff $ \Sigma $ has a transitive point. Equivalently, 
   $ \Sigma $ is {irreducible} if for every ordered pair of words $ u,\,v\in \mathcal{L} (\Sigma) $ there is a word $ w\in \mathcal{L} (\Sigma) $ so that $ uwv\in \mathcal{L} (\Sigma) $. 
   
   If forbidden set of $\Sigma$ is finite, then $\Sigma$ is
   called a \emph{shift of finite type} (SFT) and any of its factor is called \emph{sofic}.  A word $ w\in \mathcal{L}(\Sigma) $ is called \emph{synchronizing} if $ uwv \in\mathcal{L}(\Sigma) $ whenever $uw,wv \in\mathcal{L}(\Sigma)$. A \emph{synchronized system} is an irreducible shift which has a synchronizing word. Any sofic is synchronized.
   
   A synchronized system  $\Sigma$ may be represented with a directed labeled graph called \emph{cover}. Points of $\Sigma$ is the recording of labels on infinite walks and the closure of all such points. Covers are not unique and there is a ``minimal" representation which is called \emph{Fischer cover}. 
   Any Fischer cover is right resolving: the edges starting at any vertex carry different labels. For sofics, the Fischer cover is a right resolving labeled graph with the fewest vertices among all covers. Also, a subshift is sofic iff it has a finite cover.

\section{Specification property and shadowing}\label{Sp and shadowing}

On the basis of Proposition \ref{prop closure}, we will concentrate on the study of a ``generalized" IFS given in \eqref{eq IFS}. First some definitions for adjusting the existing definitions to our notations.
\begin{defn} \label{def peridic}
	Let \ifs be an IFS.	A point $x\in X$ is periodic of period $p$   along $\sigma = \sigma _1 \sigma _2 \cdots \in \Sigma$ if for any $\ell\in \mathbb{N}$, $f_{\sigma_1\cdots \sigma _{\ell p}}(x)=x$. 	
	The point $x$ is weakly periodic along $ \sigma $, if there are infinitely many $i$ such that 
		$ f_{\sigma_1\cdots \sigma _{i}}(x)=x$.
	
\end{defn}  
Later in Section \ref{sec Periodicity in IFS }, we will consider another form of periodicity in NDS; nonetheless, the above definition suffices our needs for time being.
\begin{defn}\label{defn Spi}
	Let $\mathfrak{I}=(X,\,\mathcal{F},\,\Sigma)$  be an IFS. Then $ \mathfrak{I} $ is said to have
	\begin{enumerate}
		\item 
		\emph{Specification property (SP)} along $\sigma\in\Sigma$, if for every $ \epsilon >0 $ there is $ N(\epsilon) $ such that for 
		any $ x_1,\,x_2,\,\ldots ,\,x_s $  with $ s\geq 2 $ and any sequence $ 0=j_1\leq k_1<j_2\leq k_2 < \cdots < j_s\leq k_s $ of integers with $ j_{n+1}-k_n\geq N(\epsilon) $, $ n=1,\,2,\,\ldots,\,s-1 $, 
		there is a point $ x\in X $ such that for $ 1\leq m\leq s$ and 
		$j_m\leq i\leq k_m$
		\begin{equation}\label{eq mainSP}
			d(x,\,x_1)<\epsilon, \quad	d(f_{\sigma_{1}\cdots \sigma_i}(x),\,f_{\sigma_{1}\cdots \sigma_i}(x_m))<\epsilon.
		\end{equation}  
		\item
		\emph{Strong specification property (SSP)} along $\sigma$,
		if $x\in X$ satisfying \eqref{eq mainSP} is periodic.  
		\item\label{defn local Spi}
		\emph{ Local weak specification property (LWS)}
		along $\sigma\in\Sigma$, if for every $ \epsilon>0 $  there are $ N\in \N $ and $ \delta>0 $ such that if $ x_1,\,\ldots ,\, x_s \in X$ satisfies $ d(f_{\sigma_1\cdots\sigma_n}(x_k),\,x_{k+1})$ $<\delta $ with $ n\geq N $, then there is $ x\in X $ such that for  $0\leq k \leq s-1$ and any   $kn\leq j< (k+1)n$		
		\begin{equation}\label{eq mainSP2}
			d(x,\,x_1)<\epsilon, \quad d(f_{\sigma_1\cdots\sigma_j}(x),\,f_{\sigma_1\cdots\sigma_{j-kn}}(x_{k+1}))<\epsilon.
		\end{equation} 
	\end{enumerate} 
	An IFS has SP, SSP or LWS, if it has SP, SSP or LWS respectively along all $\sigma\in\Sigma$. 
	
\end{defn} 
Definitions of SP and LWS are adapted from \cite[Definition 2.6]{salman2020specification} and \cite[Definition 3.2]{vasisht2021specification} respectively and surely \eqref{eq mainSP} as well as \eqref{eq mainSP2} can be written in one equation, but this is done for compatibility of our notations to the rest of this note.
	\begin{defn}\label{defn along an orbit}
	Let \ifs  be an IFS and $U$, $V$ arbitrary non-empty open sets in $X$.
	Then $\mathfrak{I}$ is called \emph{mixing} (resp. \emph{exact}) \emph{along an orbit}  $\sigma\in\Sigma$, if there is $N\in\N$ such that for $n\geq N$, $f_{\sigma_1\cdots\sigma_n}(U)\cap V\neq\emptyset$ (resp. $f_{\sigma_1\cdots\sigma_n}(U)=X$).
\end{defn}

 By \cite[Theorem 2.1]{memarbashi2014notes} and similar to conventional dynamical systems we have 
\begin{prop}\label{prop spec mix}
	For a surjective $\mathfrak{I}=(X,\,\mathcal{F},\,\Sigma)$, specification property implies mixing along any orbit.
\end{prop}

\begin{rem}\label{Ex mix not spec}
	\begin{enumerate}
		\item 
		Any conventional mixing system which  does not have specification property, shows that the converse of the above proposition is not necessarily true.
		\item
		Let $X=\{a,\,b\}$ and $f:X\rightarrow X$ a continuous function defined as $f(a)=f(b)=a$. Then this conventional system has specification property, though it is not mixing. Thus this simple example shows that surjectivity in Proposition \ref{prop spec mix} cannot be omitted.
			\end{enumerate}
\end{rem}

Next we will show that when $\Sigma$ is irreducible, specification property for $\mathfrak{I}$ is equivalent to having this property just along  one transitive point in $\Sigma$;
{in particular, the definition
	 of specification property can then be reduced to specification along a transitive $t\in\Sigma_{|\mathcal{F}|}$ 
	 \cite[Definition 4.2]{ghane2019topological}.
}

\begin{prop}\label{prop tranSP}
	Assume $\mathfrak{I}=(X,\,\mathcal{F},\,\Sigma)$ is surjective and has (strong) specification property along a transitive $t=t_1t_2t_3\cdots\in\Sigma$. Then $\mathfrak{I}$ has (strong) specification property along any orbit.
\end{prop}
\begin{proof}
	We will   prove that SP holds along any  $\sigma=\sigma_1\sigma_2\cdots\in \Sigma$. Similar proof works for SSP as well.
	
	For a given $\epsilon>0$, let $N=N(\epsilon)$ be the natural number  provided by the definition of specification property along $t$. We will show that this $N$ works as well for the required definition of specification property along $\sigma$.
	To this end fix $\sigma\in \Sigma$ and choose any sequence $x_1,\,x_2,\ldots, x_s$ in $X$ for $s\geq 2$ and consider the arbitrary sequence
	\begin{equation}\label{eq arbSeq}
		0=j_1\leq k_1<j_2\leq k_2 < \cdots < j_s\leq k_s 
	\end{equation}
	where $|j_{i+1}-k_i|\geq N$.
	Then choose $\ell\in\N$ so that
	$k'_0=\ell-N\geq 0$ and $t_{\ell+j}=\sigma_j$
	for $0\leq j\leq k_s$. This is possible, for $t$ is a transitive point in $\Sigma$. 
	Replace \eqref{eq arbSeq} with
	\begin{equation}\label{eq arbSq2}
		0=j'_0\leq k'_0< j'_1\leq k'_1<j'_2\leq k'_2 < \cdots < j'_s\leq k'_s 
	\end{equation}
	where $j'_i=j_i+\ell$ and $k'_i=k_i+\ell$ for $1\leq i\leq s$.
	Set  {$u=t_1 t_2 \cdots t_{\ell-1}$} and pick $y_i\in X$ so that $f_u(y_i)=x_i$, $1\leq i\leq s$ and let $y_0$ be any point in $X$.
	
	Now apply the specification property along $t$ for
	the above $\epsilon>0$, the sequence $y_0,\,y_1,\,y_2,\ldots,y_s$ and sequence \ref{eq arbSq2} to have a point $x'\in X$ such that
	\begin{equation*}
		d(x',\,y_0)<\epsilon,\quad	d(f_{t_{1}\cdots t_i}(x'),\,f_{t_{1}\cdots t_i}(y_m))<\epsilon;\quad   0\leq m\leq s,\quad  j'_m\leq i\leq k'_m.
	\end{equation*}
	This in turn implies that
	\begin{equation*}
		d(f_{\sigma_{1}\cdots \sigma_i}(f_u(x')),\,f_{\sigma_{1}\cdots \sigma_i}(f_u(y_m))<\epsilon;\quad 1\leq m\leq s ,\quad   j_m\leq i\leq k_m.
	\end{equation*}
	Hence for $x=f_u(x')$, we have   
	\begin{equation*}
		d(x,\,x_1)<\epsilon,\quad(f_{\sigma_{1}\cdots \sigma_i}(x),\,f_{\sigma_{1}\cdots \sigma_i}(x_m))<\epsilon 
	\end{equation*}
	for each $ 1\leq m\leq s $ and any $  j_m\leq i\leq k_m $. Thus specification property holds along $\sigma$ and
	by the fact that $\sigma$ was arbitrary, $\mathfrak{I}$ has specification property  as required.
\end{proof}

Thus if $\mathfrak{I}=(X,\,\mathcal{F},\,\Sigma)$ has SSP along a transitive $t\in\Sigma$, then the periodic points are dense along any $\sigma\in \Sigma$ and in particular, the dynamics along all orbits have Devaney chaos \cite{tian2006chaos}. Equivalently,
\begin{cor}
	If a surjective NDS such as \eqref{eq NDS} has SSP, then $f_{\sigma_1,\,\infty}$ has Devaney chaos for any $\sigma$  in the limit set of ${\mathcal{O}^+(t)}$.
\end{cor}
\begin{proof}
	This follows from Proposition \ref{prop tranSP} and the fact that periodic points are dense in systems with SSP and also the fact that such systems are mixing and thus transitive.
\end{proof}

 So far we have shown that 
	 specification and strong specification along a transitive orbit implies the same property along all other orbits. The rest of this section is devoted to prove similar  results for shadowing and LWS.

\begin{defn}
	The sequence $ \{x_0,\,x_1,\, \ldots\} \subseteq X $ is a\emph{ $ \delta $-pseudo orbit along $ \sigma=\sigma_1\sigma_2\cdots\in \Sigma $} for some $ \delta>0 $, if $ d(f_{\sigma_i}(x_{i-1}),\,x_i)<\delta $ for all $ i\geq 1 $. We say that \ifs has the \emph{shadowing property along $ \sigma $} if for every $ \epsilon >0 $ there exists $ \delta=\delta(\epsilon) $ such that for every $ \delta $-pseudo orbit $ \{x_0,\,x_1,\,\ldots\} $ along $ \sigma $ there is $ y\in X $ such that 
	\begin{equation*}
		d(y,\,x_0)<\epsilon, \quad d(f_{\sigma_1\cdots\sigma_i}(y),\,x_i)<\epsilon, \quad\text{ for all } i\geq 1 .
	\end{equation*}
\end{defn}

\begin{prop}\label{prop shadowing} 
	Assume $\mathfrak{I}=(X,\,\mathcal{F},\,\Sigma)$ is surjective and has shadowing along a transitive $t\in\Sigma$. Then $\mathfrak{I}$ has shadowing along any orbit.
\end{prop}

\begin{proof}
	
		Let $\sigma=\sigma_1\sigma_2\cdots$ be arbitrary
		and for $n\in\N$, let 
	$\zeta_n=\sigma_1\sigma_2\cdots\sigma_n$. Since $t$ is transitive and any $\zeta_n$ appears infinitely many times along $t$, one may write $t$ as
	\begin{equation}\label{eq t}
		t=t_1\cdots \zeta_1\cdots\zeta_2\cdots\zeta_n\cdots\zeta_{n+1}\cdots=
		t_1\cdots \zeta_1\cdots\zeta_2\cdots\sigma_1\sigma_2\cdots\sigma_n\cdots\zeta_{n+1}\cdots.
	\end{equation}
	
	Set $\underline{t}_n=t_1\cdots t_{|\underline{t}_n|}$ to be the initial finite block  of $t$ before $\zeta_n$ and $\bar{t}_n$ the infinite block after $\zeta_n$.
	Thus 
	$$t=t_1\cdots\zeta_n\cdots=\underline{t}_n\zeta_n\bar{t}_n.$$
	Let $\epsilon>0$ and
	let  $\delta>0$ be provided by the definition of $\frac{\epsilon}{ 2}$-shadowing for $t$ and let $\alpha:=\{x_i\}_{i\in\N}$ be a $\delta$-pseudo orbit along $\sigma$.
	Also for any $n$, let $\underline{x}_n$ be a point such that $f^{-1}_{\underline{t}_n}(\underline{x}_n)=x_1$.
	
	Denote by $\mathcal{S}_u(x)$  the sequence representing the true orbit of $x$ along $u$ where $u$ is a finite or infinite block in $\mathcal{L}(\Sigma)$. Then the sequence  
	$$\gamma_n:=\mathcal{S}_{\underline{t}_n}(\underline{x}_n),\,x_1,\,x_2,\,\ldots,\,x_{|\zeta_n|} ,\,\mathcal{S}_{\bar{t}_n}(x_{|\zeta_n|}) $$
	is a $\delta$-pseudo orbit along $t$. 
	Hence there is a $z_n\in X$ tracing $\gamma_n$ within an $\frac{\epsilon}{2}$ neighborhood of this {trajectory} along $t$ and set $y_n=f_{\underline{t}_n}(z_n)$. As a result we will have 
	\begin{equation}\label{eq shadow}
		{d(y_n,\,x_1)<\frac{\epsilon}{2}, \quad d(f_{\sigma_1\cdots\sigma_{n}}(y_{n}),\,x_{n+1})<
			\frac{\epsilon}{2},\qquad  1<n<|\zeta_n|.}
	\end{equation}
	Let $y$ be a limit point of $\{y_n\}$ and without loss of generality assume that $\{y_n\}$ converges to $y$. Pass to the limit in \eqref{eq shadow} to have 
	\begin{equation*}
		{d(y,\,x_1)<\epsilon, \quad d(f_{\sigma_1\cdots\sigma_{n}}(y),\,x_{n+1})<\epsilon,\qquad  n>1.}
	\end{equation*}
	Now this implies that $\alpha$ is $\epsilon$-shadowed by  $y$ along $\sigma$.
\end{proof}

	There are examples that shadowing occurs along  orbits which are not transitive. 
\begin{example}
Let $\mathfrak{I}=(X=[0,\,1],\,\{f_0,\,f_1\},\,\Sigma_F)$ where $f_0(x)=2x\mod 1$ and $f_1$ any homeomorphism say $f_1(x)=x$.
	Then if $\sigma$ terminates at $1^\infty$  i.e. if
	$\sigma=\sigma_1\cdots\sigma_n1^\infty$ for some $n\in\N$, then $\mathfrak{I}$ along $\sigma$ does not have shadowing property while  along all those $\sigma\in\Sigma$ terminating at $0^\infty$, $\mathfrak{I}$ has shadowing property. 
\end{example}

	Authors in  \cite[Theorem 4.1]{vasisht2021specification} prove that LWS along $ \sigma $ is equivalent to shadowing property along $ \sigma $. This implies the following.
\begin{cor}
	Let \ifs be a surjective IFS. Then LWS  along a transitive $t\in\Sigma$ implies LWS along any $\sigma\in\Sigma$.
\end{cor}

Consider the set of properties 
\begin{equation}\label{eq Gamma} 
	\Gamma=\{\text{SP },\text{SSP }, \text{LWS}, \text{ shadowing}\}
\end{equation}
and let $t\in\Sigma$. Then proof of propositions \ref{prop tranSP} and \ref{prop shadowing} allows us to restate those propositions as: If \ifs has property $\gamma\in\Gamma$ along $t$, then $\mathfrak{I}$ has property $\gamma$ along $\sigma$ in the limit set of $\mathcal{O}^+(t)$.
In particular, since we never used the property that $\Sigma$ is generated by a finite set of characters, thus we have
\begin{cor}
Let $\Gamma$ be as in \eqref{eq Gamma}, $f_{t_1,\,\infty}$ as in \eqref{eq NDS} defined with possibly infinitely many surjective functions $f_{t_i}:X\to X$ having property $\gamma\in\Gamma$. Then, $f_{\sigma_1,\,\infty}$  has property $\gamma$ for any $\sigma$ in the limit set of $\mathcal{O}^+(t)$.
\end{cor}

 For $\gamma\in\Gamma$ in \eqref{eq Gamma} define 
$$S_\gamma=S_\gamma(\mathfrak{I})=\{\sigma\in\Sigma:\;\mathfrak{I}\text{ has }\gamma \text{ property along }\sigma\}.$$
\begin{prop}\label{prop all3}
	Assume \ifs  is a surjective IFS and $\Sigma$ an irreducible subshift. 
	Then, either $S_\gamma=\Sigma$ or $S_\gamma$ is a meager set.
\end{prop}
\begin{proof}
	The proof is  a result of propositions \ref{prop tranSP}, \ref{prop shadowing} and the fact that the transitive points in $\Sigma$ are generic and intersection of two generic sets is again generic.
\end{proof}

In the following example, we show that the same conclusion in propositions \ref{prop tranSP} and \ref{prop shadowing} does not hold if one considers exact or mixing.

\begin{example}
	Let $\mathfrak{I}=([0,\,1],\,\{f_0,\,f_1\},\,\Sigma_2)$ where  $f_0$ is identity 
	and $f_1(x)=2x\mod 1$ on  $[0,\,1]$. This is not hard to see that $\mathfrak{I}$ is exact along any transitive point $t\in\Sigma_2$; however since $f_0$ is not transitive, $\mathfrak{I}$ is not exact  and so mixing 
	along $\sigma=0^\infty\in\Sigma_2$.
\end{example}

\section{Periodicity in non-autonomous systems and IFS} \label{sec Periodicity in IFS }

The notion of a periodic point in the case of a conventional dynamical system $ (X,\,f) $ is very natural and intuitive. This is not the case for an IFS or a non-autonomous system.
One may check \cite{pravec2019remarks} where a survey of the periodic points for a non-autonomous system is offered. 

The most prevalence definition of a periodic point for an NADS is the one given in  Definition \ref{def peridic}. Though this lacks the fact that if $x$ is periodic for $f_{\sigma_1,\,\infty}$, then $f_{\sigma_1}(x)$ is not necessarily periodic which is fairly  unintuitive with respect to a conventional dynamical system. 

\begin{defn}\label{def orbitdefn}
	A point $x\in X$ is
	\begin{enumerate}
		\item 
		\emph{regularly periodic}
		for $f_{\sigma_1,\,\infty}$ of period $p$, if for any $ i\geq 1 $ 
		$$ f_{\sigma_1\cdots \sigma_i}(x)=f_{\sigma_1\cdots\sigma_{i+p}} (x).$$
		\item
		\emph{orbitally periodic}
		for $f_{\sigma_1,\,\infty}$ of period $p$, 
		if there is a non-empty word $u$, $|u|=p$ such that $\sigma=u^\infty$ and
		$x$ is regularly periodic of period $p$.
		Thus $x$ is orbitally periodic for  $f_{\sigma_1,\,\infty}$ iff $(\sigma,\,x)\in{\bf S}$ is a periodic point of ${\bf F}$ in \eqref{eq bf S}.
	
	\end{enumerate}
	
\end{defn}\begin{center}

\end{center}
We have the following implications
\begin{eqnarray*}
	\text{orbitally periodic }&\Rightarrow& \text{ regularly periodic }\Rightarrow 
	\text{ periodic (Definition \ref{def peridic}) }\\
	&\Rightarrow& \text{ weakly periodic  (Definition \ref{def peridic})}.
\end{eqnarray*}
It is not hard to give examples to show that none of the above implications are reversible.
We give an example for one of them and the others can be also easily constructed.

\begin{example}
	Let $X$ be the unit circle and $\mathcal{F}=\{f_0,\,f_1\}$ where $f_0(x)$ is the rotation by an irrational $2\theta$ and $f_1(x)$ is the rotation by $-\theta$ and
	$$\sigma=011\cdots \overbrace{00\ldots0}^\text{$n$ times}
	\overbrace{11\ldots\ldots 1}^\text{$2n$ times}
	00\cdots.$$
	Then any point is weakly periodic but not periodic.
\end{example}

For all sorts of periodicity defined in  definitions \ref{def peridic} and  \ref{def orbitdefn}, the eventual periodicity  can be defined accordingly. For instance, a point $x\in X$ is eventually weakly periodic if there is $i$ such that 
$y=f_{\sigma_1\cdots\sigma_i}(x)$ is weakly periodic.

\begin{prop}
	Suppose that $x$ is periodic for $f_{\sigma_1,\,\infty}$ where $\sigma_i\in \mathcal{A}=\{0,\,1,$ $\,\ldots,k-1\}$.
	Then for any $j$, $y=f_{\sigma_1\cdots\sigma_j}(x) $ is eventually weakly periodic.
\end{prop}
\begin{proof} 
	Let $x$ be periodic of period $p$ and for some $q\in \N_0=\N\cup\{0\}$
	write $j=qp+\ell$, 
	$0\leq \ell< p$. Then the set
	$\{f_{\sigma_{ip+1}\cdots\sigma_{ip+\ell}}(x):\; i\geq q\}$ is finite; because,
	$\{\sigma_{ip+1}\cdots\sigma_{ip+\ell}:\;i\geq q\}\subset
	\mathcal{L}_{\ell}(\Sigma_k)$ is finite 
	and thus for infinitely many  $i$, the values of 
	$f_{\sigma_{ip+1}\cdots\sigma_{ip+\ell}}(x)$ are identical.
\end{proof}

\begin{lem}\label{lem: PerSFT}
	Let $\mathfrak{I}=(X,\,\mathcal{F},\,\Sigma)$ be an IFS where $\Sigma$ is an $M$-step SFT. 
	Assume that $x$ is periodic of period $p$, then there is $\sigma'={u'}^\infty\in\Sigma$ where $x$ is orbitally periodic for $f_{\sigma'_1,\,\infty}$ with
	$|u'|=qp$ for some $q\in\N$.
\end{lem}
\begin{proof}
	Without loss of generality assume that $p\geq M$;  otherwise, we may choose $a\in\N$ so large with $a p\geq M$ and considering $x$ as a periodic point of period $ap$.
	
	Set $u_i:=\sigma_{(k_ip)+1}\sigma_{(k_ip)+2}\cdots \sigma_{k_{i+1}p}$ and note that $f_{u_i}(x)=x$. 
	But our alphabet is finite, thus for infinitely many $i$'s, $u_i$'s are the same and call it $u$.
	This means that there is a $w$ with $uwu\subset\sigma$ and $|w|$ an integral multiple of $p$.
	Now by the fact that $u$'s length is larger than $M$, it is synchronizing and so $\sigma'=(uw)^{\N_0}\in\Sigma$ and that in turn implies that $x$ is periodic of period $p'=|uw|$ along $\sigma'$.
\end{proof}
Let $\Sigma$ and $\Sigma'$ be two subshifts on the character sets $\mathcal{A}$ and $\mathcal{A}'$ respectively. Then, $\Sigma'$ is a factor of $\Sigma$ or $\Sigma$ is an extension of $\Sigma'$ if there is a map $\Psi:\mathcal{L}_{m+n+1}(\Sigma)\to \mathcal{A}'$ called the \emph{block map}   of memory $m$ and anticipation $n$ and a  \emph{code map} $\psi=\Psi_\infty:\Sigma\to \Sigma'$ such that 
$\psi(\sigma)_i=\sigma'_i=\Psi(\sigma_{i-m}\cdots\sigma_{i+n})$.
\begin{defn}\cite{dastjerdi2022iterated}\label{def factor}
		Let $ \varphi $ be a continuous map from $ X $ onto $ Y $. Then $ \mathfrak{I}'=(Y,\mathcal{G}=\{g_0,\,\ldots,\,g_{\ell'}\},\,\Sigma') $ is a \emph{factor} of $\mathfrak{I}=(X,\mathcal{F}=\{f_0,\,\ldots,\,f_{\ell}\},\,\Sigma) $ if $\Sigma'$ is a factor of $\Sigma$ as above and for all $ y\in Y $ and $ x\in \varphi^{-1}(y) $
		\begin{equation*}
			\varphi\circ f_{\sigma_{i-m}\cdots \sigma_{i+n}}(x)=g_{\sigma_i'}(y) .
			\end{equation*}
	\end{defn}

\begin{rem}
	Let $({\bf S}_{\mathfrak{I}}, \, {\bf F}_{\mathfrak{I}})$ be the conventional dynamical system associated to $\mathfrak{I}$ as in \S \ref{sec preliminaries}. In that case,
	if $\mathfrak{I}'$ is a factor of $\mathfrak{I}$ given in Definition \ref{def factor}, then $({\bf S}_{\mathfrak{I}'}, \, {\bf F}_{\mathfrak{I}'})$ is a factor of $({\bf S}_{\mathfrak{I}}, \, {\bf F}_{\mathfrak{I}})$.  Surely, there are many other factors of 
	$({\bf S}_{\mathfrak{I}}, \, {\bf F}_{\mathfrak{I}})$ which do not look like this. Recall that we have this situation for subshifts: many factors of a subshift are not necessarily  subshifts.
\end{rem}

\begin{prop}\label{prop: PerSofic}
	Let $\mathfrak{I}$ and $x\in X$ be as in Lemma \ref{lem: PerSFT}, but $\Sigma$ an irreducible sofic. Then the  conclusion of that lemma  is valid.
\end{prop}
\begin{proof}
	Let $G$ be the Fischer cover (minimal right resolving graph representing the subshift) of $\Sigma$. Then, $G$ is a finite graph for our sofic shift  and let
	$\Sigma_G$  be the associated edge shift (the subshift arising from the infinite walk on $G$ and recording labels after assigning different labels to different edges). 
	This $\Sigma_G$ is a 1-step SFT and an extension of $\Sigma$. Thus by letting $\varphi: X\rightarrow X$ to be the identity map, $\mathfrak{I}$ is a factor of $\mathfrak{I}'=(X,\,\mathcal{F},\,\Sigma_G)$ and $x$ is a periodic point of $\mathfrak{I}'$ as well.  Now apply Lemma \ref{lem: PerSFT} and recall that periodic points are preserved by factor maps. 
\end{proof}

Recall that SFT's $\subsetneq$ sofics $\subsetneq$ synchronized systems. 
\begin{example}\label{ex Morse}
	The conclusion of Proposition \ref{prop: PerSofic} is not valid for the case when $\Sigma$ 
	is a non-sofic synchronized system.
\end{example}
\begin{proof}
	Let $\Sigma$ be the synchronized subshift whose cover is presented in Figure \ref{Fig: Morse}.
	In this cover, $m=m_0m_1\cdots$ is a fixed point of the Morse substitution and thus $\Sigma_m=\overline{\mathcal{O}^+(m)}$ is a minimal subshift \cite{morse1938symbolic}. Let us briefly remind the Morse substitution. Set 
	\[
	\varrho(v)=
	\begin{cases}
		01,\quad \text{if }v=0,\\
		10,\quad \text{if }v=1,
	\end{cases}
	\]
	to be the substitution map which means for any word  $u=u_0\cdots u_k$ in $\{0,\,1\}^\N$, $\varrho(u) =\varrho(u_0)\cdots \varrho(u_k)$. This 
	gives a primitive substitution with two fixed points
	\begin{eqnarray*}
		m&=0\mapsto 01\mapsto 0110\mapsto 01101001\mapsto\cdots\\
		m'&=1\mapsto 10\mapsto 1001\mapsto 10010110\mapsto\cdots
	\end{eqnarray*}
	Now, the closure of orbits of either of these fixed points under the shift map gives a minimal subshift. 
	
	Let $\mathfrak{I}=(S^1,\,\mathcal{F}=\{f_0,\,f_1,\,f_2\},\,\Sigma)$ where $f_0:S^1\rightarrow S^1$ is an irrational rotation, $f_1=f_0^{-1}$ and $f_2$ another irrational rotation say $f_2=f_0\circ f_0$.
	Since $\Sigma_m=\overline{\mathcal{O}^+(m)}$  is minimal, the only periodic points of $\Sigma$ are the cycles presented in the cover and in particular $m$ is not periodic;
	however $f_m(s)$, the orbit of $s$ along $m$, is periodic of period two for all $s\in S^{1}$. 
	In fact, if $\sigma=\sigma_0\sigma_1\cdots \in\Sigma$ is any periodic point (or any point with $\sigma_i=2$ for infinitely many $i$),
	then the orbit of any $s\in S^{1}$ is transitive along $\sigma$.
	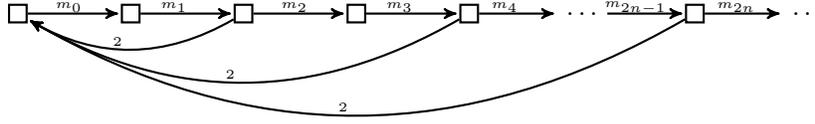
\begin{figure}
		\begin{center} 
			\begin{tikzpicture}[->,>=stealth',shorten >=1pt,auto,node distance=1.5cm,
				thick,main node/.style={rectangle,draw,font=\sffamily\small\bfseries}]
				
				\node[main node](1) {};
				\node[main node] (2) [ right of=1] {};
				\node[main node] (3) [right of=2] {};
				\node[main node] (4) [right of=3] {};
				\node[main node] (5) [right of=4] {};
				\node[] (6) [ right of=5] {$\ldots$};
				\node[main node] (7) [right of=6] {};
				\node[] (8) [ right of=7] {$\ldots$};

				\path[every node/.style={font=\sffamily\small}]
				(1) edge node [right] { $\hspace{-6mm}\ ^{\ ^{m_{0}}}$} (2) 
				(2) edge node [right] {$\hspace{-7mm}\ ^{\ ^{m_{1}}}$} (3) 
				(3) edge node [right] {$\hspace{-6mm}\ ^{\ ^{m_{2}}}$} (4)
				edge [bend left] node[right] {$\hspace{-6mm}\ ^{\ ^{2}}$} (1)
				
				(4) edge node [right] {$\hspace{-7mm}\ ^{\ ^{m_{3}}}$} (5)
				(5)    edge node [right] {$\hspace{-7mm}\ ^{\ ^{m_{4}}}$} (6)
				edge [bend left] node[right] {$\hspace{-6mm}\ ^{\ ^{2}}$} (1)
				
				(6) edge node [right] {$\hspace{-9mm}\ ^{\ ^{m_{2n-1}}}$} (7)
				
				(7) edge node [right] {$\hspace{-7mm}\ ^{\ ^{m_{2n}}}$} (8)
				edge [bend left] node[right] {$\hspace{-6mm} \ ^{\ ^{2}} $} (1);
			\end{tikzpicture}
		\end{center}
		\begin{center}
			\caption{The cover for  $ \Sigma$. Here, $m=m_0m_1\cdots$ is a fixed point of the Morse substitution.}\label{Fig: Morse} 
		\end{center}
	\end{figure}
\end{proof}

\begin{prop}\label{prop transPer}
	Let \ifs be an IFS.
	If $x\in X$ is periodic of period $p$ along a transitive $t=t_1t_2\cdots\in\Sigma$,  then for any $\sigma\in\Sigma$ there is a $1\leq k<p$ such that $x$ is periodic along $\sigma'=\tau^k(\sigma)$ where $\tau$ is our shift map. 
\end{prop}
\begin{proof}
	Let $\sigma=\sigma_1\sigma_2\cdots$ and for $n\in\N$, let 
	$\zeta_n=\sigma_1\sigma_2\cdots\sigma_n$. 
	Write $t$ as in \eqref{eq t}.
	%
	%
	If $n>p$, then by the fact that for $\ell\in\N$, $f_{t_1\cdots t_{\ell p}}(x)=x$, there is $1\leq i<p$ such that for some $\ell$ 
	\begin{eqnarray}\label{eq t=t_1...}
		t_1\cdots t_{\ell p}=&t_1\cdots\zeta_1\cdots\zeta_2\cdots\sigma_1\sigma_2\cdots\sigma_i,
		\quad\text{and}	\nonumber\\
		f_{t_1\cdots t_{\ell p}}(x)=&\!\!\!\!\!\!\!\!\!\!\!\!\!\!\!\!\!\!\!
		f_{t_1\cdots\zeta_1\cdots\zeta_2\cdots\sigma_1\sigma_2\cdots\sigma_i}(x)=x.
	\end{eqnarray}
	The word $\sigma_1\sigma_2\cdots\sigma_i$ appearing in \eqref{eq t=t_1...} is the initial segment of $\zeta_n$ of length $i\;(<p<n)$.
	Then since we have infinitely many $\zeta_n$, infinitely many of $i$'s in \eqref{eq t=t_1...} must be identical and denote that by $k$ and the corresponding $\zeta_n$ by $\zeta_{k_i}$.
	Now along $\zeta_{k_i}$ as a subword of $t$ and for any $\ell\in\N$ where $ k+\ell p\leq k_i$, we have 
	$f_{\sigma_{k+1}\cdots\sigma_{k+\ell p}}(x)=x$ and the conclusion follows by an induction argument and the fact that $k_i\to\infty$.
\end{proof}

Example \ref{ex Morse} is a case where the IFS has periodic points but not along any transitive $t\in\Sigma$. 

As an application of Proposition \ref{prop transPer}, consider $\mathfrak{I}=(S^1,\,\{f_0,\,f_1\},\,\Sigma_F)$ where $f_0$ is an irrational rotation and $f_1(z)=z^2$ and $\Sigma_F$ the full shift and hence sofic. Here, $\mathfrak{I}$ has a set of dense periodic points along $\sigma=1^\infty\in\Sigma$ and no periodic points along $0^\infty\in\Sigma$ and so no periodic points along any transitive $t\in\Sigma$.

Let us end this note with a remark on considering the classical IFS as a conventional dynamical system.
\begin{rem}
	As it was pointed out earlier in \S
	\ref{sec preliminaries}, one may consider an IFS as a conventional dynamical system $({\bf S},\,{\bf F})$ where ${\bf F}$ is defined  in \eqref{eq bf S}.
	In this situation, if our subshift is a full shift $\Sigma_{|\mathcal{F}|}$, 
	then dynamical properties such as transitivity, mixing, having shadowing for $\mathfrak{I}=(X,\,\mathcal{F},\,\Sigma_{|\mathcal{F}|})$ is characterized by the respective properties for  $({\bf S},\,{\bf F})$ and vice versa. 
	In virtue of Proposition \ref{prop: PerSofic}, this claim is even true for a point $x$ being periodic.	However, from point of dynamical considerations, a  general IFS \ifs and $({\bf S},\,{\bf F})$ are two different entities. 
\end{rem}



\begin{thebibliography}{99}
	\bibitem{dastjerdi2022iterated}
	Dawoud Ahmadi Dastjerdi and Mahdi Aghaee. \textit{Iterated function systems over arbitrary shift spaces.} arXiv preprint arXiv:2203.15264, 2022.

	
	\bibitem{ghane2019topological}
	 Fatemeh H Ghane and Javad Nazarian Sarkooh. \textit{On topological entropy and topological pressure of non-autonomous iterated function systems.} Journal of the Korean Mathematical Society 56(6):1561-1597, 2019.
	
		
	\bibitem{kolyada1996topological}
	Sergiĭ Kolyada  and Lubomir Snoha. \textit{Topological entropy of nonautonomous dynamical systems.} Random and computational dynamics, 4(2):205, 1996.
	
	
	\bibitem{lind2021introduction}
	Douglas Lind  and Brian Marcus. An introduction to symbolic dynamics and coding. Cambridge university press, 2021.
	\bibitem{memarbashi2014notes}
Reza  Memarbashi and Hosein Rasuli. \textit{Notes on the dynamics of nonautonomous discrete dynamical systems.} J. Adv. Res. Dyn. Control Syst, 6(2):8-17, 2014.
		\bibitem{morse1938symbolic}
	 Marston Morse and Gustav A. Hedlund. \textit{Symbolic dynamics.} American Journal of Mathematics 60(4):815-866, 1938.
		\bibitem{pravec2019remarks}
		Vojtech Pravec. \textit{Remarks on definitions of periodic points for nonautonomous dynamical system.} Journal of Difference Equations and Applications, 25(9-10):1372-1381, 2019.
		\bibitem{rodrigues2016specification}
		Rodrigues Fagner B and Paulo Varandas. \textit{Specification and thermodynamical properties of semigroup actions.} Journal of Mathematical Physics, 57(5):052704, 2016.
		\bibitem{salman2020multi}
	Mohammad Salman  and Ruchi Das. \textit{Multi-transitivity in non-autonomous discrete systems.} Topology and its Applications, 278:107237, 2020.
		\bibitem{salman2020specification}
		Mohammad Salman  and Ruchi Das. \textit{Specification properties for non-autonomous discrete systems.} Topological Methods in Nonlinear Analysis, 55(2):475-491, 2020.
		\bibitem{sanchez2017chaos}
	Ivan	Sanchez, Manuel Sanchis and Hugo Villanueva. \textit{Chaos in hyperspaces of nonautonomous discrete systems.} Chaos, Solitons \& Fractals, 94:68-74, 2017.
		\bibitem{tian2006chaos}
		Chuanjun Tian  and Guanrong Chen. \textit{Chaos of a sequence of maps in a metric space.} Chaos, Solitons $\&$ Fractals, 28(4):1067-1075 (2006).
	\bibitem{vasisht2020furstenberg}
	 Radhika Vasisht and Ruchi Das. \textit{Furstenberg families and transitivity in non-autonomous systems.} Asian-European Journal of Mathematics, 13(01):2050029 2020.
	
	\bibitem{vasisht2021specification}
	Radhika Vasisht  and Ruchi Das. \textit{Specification and shadowing properties for non-autonomous systems.} Journal of Dynamical and Control Systems, pages 1-12, 2021.
	
		
	\bibitem{bahabadi2015shadowing}
	Alireza Zamani Bahabadi. \textit{Shadowing and average shadowing properties for iterated function systems.} Georgian Mathematical Journal, 22(2):179-184, 2015.

	
	
	
\end{thebibliography}

\end{document}